\documentclass[11pt]{amsart}
\pdfoutput=1

\usepackage{amssymb}
\usepackage{microtype}
\usepackage{mathtools}
\usepackage{enumitem}

\usepackage[hmargin=1.25in,vmargin=1in,marginparwidth=1in]{geometry}

\usepackage[numbers]{natbib}

\usepackage{hyperref}
\hypersetup{
  pdftitle={Hamiltonian mechanics and Lie algebroid connections},
  pdfauthor={Jiawei Hu and Ari Stern},
  pdfsubject={MSC 2020: 53D17},
  bookmarksopen=false
}

\usepackage[capitalize,nameinlink]{cleveref}

\makeatletter
\g@addto@macro\bfseries{\boldmath}
\makeatother

\theoremstyle{plain}
\newtheorem{theorem}{Theorem}[section]
\newtheorem{lemma}[theorem]{Lemma}
\newtheorem{corollary}[theorem]{Corollary}

\theoremstyle{definition}
\newtheorem{definition}[theorem]{Definition}
\newtheorem{example}[theorem]{Example}

\theoremstyle{remark}
\newtheorem{remark}[theorem]{Remark}

\begin{document}

\title{Hamiltonian mechanics and Lie algebroid connections}

\author{Jiawei Hu}
\address{Department of Mathematics, University of Bonn}
\email{s34jhu@uni-bonn.de}

\author{Ari Stern}
\address{Department of Mathematics and Statistics,
  Washington University in St.~Louis}
\email{stern@wustl.edu}

\begin{abstract}
  We develop a new, coordinate-free formulation of Hamiltonian
  mechanics on the dual of a Lie algebroid. Our approach uses a
  connection, rather than coordinates in a local trivialization, to
  obtain global expressions for the horizontal and vertical
  dynamics. We show that these dynamics can be obtained in two
  equivalent ways: (1) using the canonical Lie--Poisson structure,
  expressed in terms of the connection; or (2) using a novel
  variational principle that generalizes Hamilton's phase space
  principle.
\end{abstract}

\maketitle

\section{Introduction}

The dual $\mathfrak{g}^{\ast}$ of a Lie algebra $\mathfrak{g}$ has a
canonical (up to sign) \emph{Lie--Poisson structure}, defined by the
$ ( \pm ) $ Lie--Poisson brackets
\begin{equation}
  \label{e:algebra_bracket}
  \{ F, G \} _\pm (\mu) \coloneqq \pm \Biggl\langle \mu , \biggl[ \frac{ \delta F }{ \delta \mu } , \frac{ \delta G }{ \delta \mu } \biggr]\Biggr\rangle.
\end{equation}
Here, $ F, G \colon \mathfrak{g}^{\ast} \rightarrow \mathbb{R} $,
$ \mu \in \mathfrak{g}^{\ast} $, $ \langle \cdot , \cdot \rangle $ is
the duality pairing between $ \mathfrak{g}^{\ast} $ and
$\mathfrak{g}$, $ [ \cdot , \cdot ] $ is the Lie bracket on
$ \mathfrak{g} $, and $ \delta F / \delta \mu \in \mathfrak{g} $ is
the variational derivative defined by
$ \bigl\langle \delta \mu , \delta F / \delta \mu \bigr\rangle = \lim
_{ \epsilon \rightarrow 0 } \bigl[ F ( \mu + \epsilon \, \delta \mu )
- F (\mu) \bigr] / \epsilon $ for all
$ \delta \mu \in \mathfrak{g}^{\ast} $. Using the $ (\pm) $ bracket, a
Hamiltonian $ H \colon \mathfrak{g}^{\ast} \rightarrow \mathbb{R} $
gives rise to the dynamics
\begin{equation}
  \label{e:lie-poisson}
  \dot{ \mu } = \mp \operatorname{ad} ^\ast _{ \delta H / \delta \mu } \mu ,
\end{equation}
which are called the \emph{Lie--Poisson equations}
(\citet{MaRa1999}). These arise in diverse applications ranging from
rigid body mechanics to incompressible fluid dynamics.

More generally, one may define a Lie--Poisson structure on the dual of
a \emph{Lie algebroid} (\citet{Courant1990, Weinstein1996}). This
includes as special cases not only \eqref{e:algebra_bracket} but also
the canonical Poisson structure on a cotangent bundle or its quotient
by a Lie group action. As such, Lie algebroid duals provide a rich
setting for Hamiltonian mechanics and reduction.  However, in contrast
to the global, coordinate-free form of \eqref{e:lie-poisson},
Hamilton's equations in this more general setting have only been
expressed with respect to coordinates in a local trivialization:
cf.~\citet*{Weinstein1996,Martinez2001,LeMaMa2005}.

In this paper, we develop a new, coordinate-free formulation of
Hamiltonian mechanics on the dual of a Lie algebroid. Our approach
uses a connection, rather than coordinates in a local trivialization,
to obtain global expressions for the horizontal and vertical
dynamics. In particular, the Lie--Poisson equations
\eqref{e:lie-poisson} are obtained as a special case, with
$ \operatorname{ad} ^\ast $ arising from the connection and
$ \delta H / \delta \mu $ corresponding to the vertical part of
$ \mathrm{d} H $. This can be viewed as a Hamiltonian counterpart to
the paper \citet*{LiStTa2017}, which provides the Lagrangian side of
this story.

The paper is organized as follows:
\begin{itemize}
\item \Cref{sec:algebroids} reviews the basic ideas of Lie algebroids,
  previous work on generalized Lie--Poisson dynamics in local
  coordinates, and Lie algebroid connections.

\item \Cref{sec:lie-poisson} develops a coordinate-free expression for
  the Lie--Poisson structure on a Lie algebroid dual relative to a
  connection, and uses this to derive the horizontal and vertical
  dynamics arising from a Hamiltonian.

\item \Cref{sec:variational} introduces a novel variational principle
  for Hamiltonian dynamics on Lie algebroid duals, which generalizes
  Hamilton's phase space principle in the case of a cotangent
  bundle. Using a connection, it is shown that paths satisfying the
  variational principle are precisely solutions to the horizontal and
  vertical Lie--Poisson equations derived in the preceding section. As
  special cases, this also generalizes the Lie--Poisson and
  Hamilton--Poincar\'e variational principles of \citet{CeMaPeRa2003}.

\item \Cref{sec:lagrangian} concludes the paper with a discussion of
  the relationship between the Hamiltonian and Lagrangian formalisms,
  linking the results of the present paper with those of
  \citet*{LiStTa2017}.
\end{itemize} 

\subsection*{Acknowledgments}

The authors wish to thank Rui Loja Fernandes for helpful discussions
during the early stages of this project. Jiawei Hu was supported in
part by the Freiwald Scholars program in the Department of Mathematics
and Statistics at Washington University in St.~Louis. Ari Stern was
supported in part by NSF grants DMS-1913272 and DMS-2208551.

\section{Lie algebroid preliminaries}
\label{sec:algebroids}

\subsection{Lie algebroids} We begin with the definition of a Lie
algebroid, along with a few standard examples. See
\citet{Mackenzie2005} for a comprehensive reference.

\begin{definition}
  A \emph{Lie algebroid} is a real vector bundle
  $ \tau \colon A \rightarrow Q $ equipped with a Lie bracket
  $ [ \cdot , \cdot ] $ on the space of sections $ \Gamma (A) $ and a
  bundle map $ \rho \colon A \rightarrow T Q $, called the
  \emph{anchor}, satisfying the Leibniz rule
  \begin{equation*}
    [ X, f Y ] = \rho (X) [f] Y + f [ X, Y ] ,
  \end{equation*}
  for all $ f \in C ^\infty (Q) $ and $ X, Y \in \Gamma (A) $.
\end{definition}

\begin{remark}
  Here and henceforth, we use the notation
  $ v [ f ] \coloneqq \langle \mathrm{d} f , v \rangle $ for a tangent
  vector or vector field acting as a derivation on smooth functions.
\end{remark}

\begin{example}
  \label{e:tangent}
  The tangent bundle $ \tau \colon T Q \rightarrow Q $ is a Lie
  algebroid, where $ [ \cdot , \cdot ] $ is the Jacobi--Lie bracket on
  vector fields and $ \rho \colon T Q \rightarrow T Q $ is the
  identity map.
\end{example}

\begin{example}
  \label{e:algebra}
  A Lie algebra $\mathfrak{g}$ can be interpreted as a Lie algebroid
  over a single point
  $ \tau \colon \mathfrak{g} \rightarrow \bullet $, where
  $ [ \cdot , \cdot ] $ is the bracket on $\mathfrak{g}$ and $ \rho $
  is trivial.
\end{example}

\begin{example}
  \label{e:atiyah}
  Given a principal $G$-bundle $ Q \rightarrow Q / G $, one may form
  the \emph{Atiyah algebroid}
  $ \tau \colon T Q / G \rightarrow Q / G $, where
  $ [ \cdot , \cdot ] $ agrees with the Jacobi--Lie bracket of
  $G$-invariant vector fields on $Q$, and where $\rho$ agrees with the
  ($G$-invariant) tangent map of the principal bundle
  projection. \Cref{e:tangent,e:algebra} are the special cases
  $ G = \{ e \} $ and $ Q = G $, respectively.
\end{example}

We are often interested in a particular class of paths in Lie
algebroids, especially for applications in geometric mechanics.

\begin{definition}
  Let $ a \colon I \rightarrow A $ be a path in $A$, where $I$ is an
  interval (interpreted as time), and let
  $ q = \tau \circ a \colon I \rightarrow Q $ be the corresponding
  base path in $Q$. We say that $a$ is an \emph{$A$-path} over $q$ if
  it satisfies $ \dot{q} (t) = \rho \bigl( a (t) \bigr) $ for all
  $ t \in I $.
\end{definition}

For example, $ T Q $-paths are just the tangent prolongations of paths
in $Q$. (In the context of geometric mechanics,
\citet{YoMa2006a,YoMa2006b} call this the ``second-order condition.'')
On the other hand, every path in $\mathfrak{g}$ is a
$\mathfrak{g}$-path, since the condition to be satisfied is trivial.

\subsection{The generalized Lie--Poisson structure on \texorpdfstring{$ A ^\ast $}{A*}}

Let $ \pi \colon A ^\ast \rightarrow Q $ denote the dual bundle of a
Lie algebroid $A$.

\subsubsection{The $ ( \pm ) $ Lie--Poisson brackets}
In order to define the Lie--Poisson structure on $ A ^\ast $, we first
introduce some useful notation adopted from \citet{Marle2008}.

\begin{definition}
  Given a section $ X \in \Gamma (A) $, let
  $ \Phi _X \in C ^\infty ( A ^\ast ) $ denote the fiberwise-linear
  function $ \Phi _X ( p ) = \langle p, X \rangle $, where
  $ \langle \cdot , \cdot \rangle $ is the duality pairing between
  $ A ^\ast $ and $A$.
\end{definition}

\begin{subequations}
  \label{e:algebroid_bracket_relations}
  \begin{definition}
    The $ ( \pm ) $ \emph{Lie--Poisson bracket} on $ A ^\ast $ is the
    unique bracket satisfying
    \begin{equation}
      \label{e:algebroid_bracket_XY}
      \{ \Phi _X , \Phi _Y \}_\pm = \pm \Phi _{ [X, Y] } ,
    \end{equation}
    for all $ X, Y \in \Gamma (A) $.
  \end{definition}

  For any $ f, g \in C ^\infty (Q) $, replacing $X$ and $Y$ in
  \eqref{e:algebroid_bracket_XY} by $ f X $ and $ g Y $ shows that, in
  order for $ \{ \cdot , \cdot \} _\pm $ to satisfy the Leibniz rule,
  it must also satisfy
  \begin{align}
    \{ \Phi _X , g \circ \pi \} _\pm &= \pm \rho (X) [g] \circ \pi , \label{e:algebroid_bracket_Xg}\\
    \{ f \circ \pi , g \circ \pi \}_\pm &= 0 \label{e:algebroid_bracket_fg}.
  \end{align}
\end{subequations}
Since every covector in $ T ^\ast _p A ^\ast $ can be written as
$ \mathrm{d} ( \Phi _X + f \circ \pi ) (p) $ for some
$ X \in \Gamma (A) $ and $ f \in C ^\infty (Q) $ \citep[Lemma
6.3.2]{Marle2008}, this completely defines the Lie--Poisson structure
on $ A ^\ast $. \citet[Theorem 2.1.4]{Courant1990} shows that the
converse also holds: any fiberwise-linear Poisson structure on a dual
vector bundle $ E ^\ast \rightarrow Q $ determines a Lie algebroid
structure on $ E \rightarrow Q $.

\subsubsection{Local coordinates and Hamiltonian dynamics}
Let $ q ^i $ be local coordinates for $Q$ and $ \{ e _I \} $ be a
local basis of sections of $A$. Let $ C _{ I J } ^K $ and
$ \rho ^i _I $ be the local-coordinate representations of
$ [ \cdot , \cdot ] $ and $\rho$, defined by
\begin{equation*}
  [ e _I , e _J ] = C _{ I J } ^K e _K , \qquad \rho ( e _I ) = \rho ^i _I \frac{ \partial }{ \partial q ^i } .
\end{equation*}
(We use the Einstein index convention, where there is an implicit sum
over repeated indices.)  In the special case of a Lie algebra,
$ C _{ I J } ^K $ are known as \emph{structure constants}; in the
general case, they depend smoothly on $ q \in Q $ and are called
\emph{structure functions}.

It follows that $ q ^i $ and $ \mu _I = \Phi _{ e _I } $ are local
coordinate functions on $ A ^\ast $, where $ \mu _I $ gives the
coefficient in the dual basis $ \{ e ^I \} $ of sections of
$ A ^\ast $. We can now express the bracket relations
\eqref{e:algebroid_bracket_relations}, in terms of these local coordinate
functions, as
\begin{align*}
  \{ \mu _I , \mu _J \} _\pm &= \pm C _{ I J } ^K  \mu _K ,\\
  \{ \mu _I , q ^i \} _\pm &= \pm \rho ^i _I ,\\
  \{ q ^i , q ^j \} _\pm &= 0 .
\end{align*}
(We commit a slight abuse of notation by identifying $ q ^i $ with
$ q ^i \circ \pi $.)  It follows that
\begin{equation}
  \label{e:algebroid_bracket_coordinates}
  \{ F, G \} _\pm = \pm C _{ I J } ^K \frac{ \partial F }{ \partial \mu _I } \frac{ \partial G }{ \partial \mu _J } \mu _K \pm \rho ^i _I \biggl( \frac{ \partial F }{ \partial \mu _I } \frac{ \partial G }{ \partial q ^i } - \frac{ \partial F }{ \partial q ^i } \frac{ \partial G }{ \partial \mu _I } \biggr) ,
\end{equation}
for arbitrary functions $ F, G \in C ^\infty ( A ^\ast ) $.

Now, given a Hamiltonian $ H \in C ^\infty ( A ^\ast) $, we recall
that Hamilton's equations are defined by
$ \dot{ F } = \{ F, H \} _\pm $ for all
$ F \in C ^\infty ( A ^\ast ) $. Taking $F$ to be the coordinate
functions $ q ^i $ and $ \mu _I $, we obtain
\begin{subequations}
  \label{e:lie-poisson_coordinates}
  \begin{align}
    \dot{q} ^i &= \mp \rho ^i _I \frac{ \partial H }{ \partial \mu _I } , \label{e:lie-poisson_coordinates_q}\\
    \dot{ \mu } _I &= \pm C _{ I J } ^K \frac{ \partial H }{ \partial \mu _J }  \mu _K \pm \rho ^i _I \frac{ \partial H }{ \partial q ^i } , \label{e:lie-poisson_coordinates_mu}
  \end{align}
\end{subequations}
which are local-coordinate expressions of the generalized Lie--Poisson
equations on $ A ^\ast $.

\begin{example}
  \label{ex:tq_poisson}
  Let $ \tau \colon T Q \rightarrow Q $ be the tangent bundle of $Q$
  and $ \pi \colon T ^\ast Q \rightarrow Q $ be the cotangent bundle.
  Local coordinates $ q ^i $ on $Q$ give rise to a local basis of
  sections $ \partial / \partial q ^i $ of $TQ$. From this, we get
  so-called canonical coordinates $ q ^i , p _i $ on $ T ^\ast Q $,
  where $ p _i = \Phi _{ \partial / \partial q ^i } $ in the notation
  used above. (Due to the correspondence between base and fiber
  coordinates, we use lowercase indices $i, j, k $ for both.) In these
  coordinates, we have
  \begin{equation*}
    C _{ i j } ^k = 0 , \qquad \rho ^i _j = \delta ^i _j ,
  \end{equation*}
  where $ \delta ^i _j $ is the Kronecker delta. Thus,
  \eqref{e:algebroid_bracket_coordinates} gives the $ ( \pm ) $
  brackets on $ T ^\ast Q $,
  \begin{equation}
    \label{e:cotangent_bracket}
    \{ F , G \} _\pm = \pm \biggl( \frac{ \partial F }{ \partial p _i } \frac{ \partial G }{ \partial q ^i } - \frac{ \partial F }{ \partial q ^i } \frac{ \partial G }{ \partial p _i } \biggr),
  \end{equation}
  with respect to which Hamilton's equations
  \eqref{e:lie-poisson_coordinates} are
  \begin{subequations}
    \label{e:cotangent_equations}
    \begin{align}
      \dot{q} ^i &= \mp \frac{ \partial H }{ \partial p _i } , \label{e:cotangent_equations_q}\\
      \dot{p} _i &= \pm \frac{ \partial H }{ \partial q ^i }. \label{e:cotangent_equations_p}
    \end{align}
  \end{subequations}
  From this, we see that the $ ( - ) $ bracket gives the usual form of
  the canonical Poisson structure and Hamilton's equations on
  $ T ^\ast Q $, while $ ( + ) $ gives the opposite sign convention.
\end{example}

\begin{example}
  Let $ \tau \colon \mathfrak{g} \rightarrow \bullet $ be a Lie
  algebra and $ \pi \colon \mathfrak{g} ^\ast \rightarrow \bullet $
  its dual, considered as vector bundles over a single point. Given a
  basis $ \{ e _I \} $ of $\mathfrak{g}$, the coefficients
  $ C _{ I J } ^K $ are the structure constants of $ \mathfrak{g} $,
  and the anchor is trivial. In these coordinates, the $ ( \pm ) $
  brackets \eqref{e:algebroid_bracket_coordinates} on
  $ \mathfrak{g}^{\ast} $ are
  \begin{equation*}
    \{ F , G \} _\pm = \pm C _{ I J } ^K \frac{ \partial F }{ \partial \mu _I } \frac{ \partial G }{ \partial \mu _J } \mu _K ,
  \end{equation*}
  with respect to which Hamilton's equations
  \eqref{e:lie-poisson_coordinates} are
  \begin{equation*}
    \dot{ \mu } _I = \pm C _{ I J } ^K \frac{ \partial H }{ \partial \mu _J } \mu _K .
  \end{equation*}
  These are precisely local-coordinate expressions for the $ ( \pm ) $
  Lie--Poisson bracket \eqref{e:algebra_bracket} and Lie--Poisson
  equations \eqref{e:lie-poisson}, respectively.
\end{example}

\subsection{Connections and variations of $A$-paths}

We recall the notion of connection on a Lie algebroid due to
\citet{CrFe2003} (see also \citet{Fernandes2002}), along with the
important role that such connections play in calculus of variations
for $A$-paths.

\begin{definition}
  \label{def:a-connection}
  Given a Lie algebroid $ A \rightarrow Q $ an \emph{$A$-connection}
  on a vector bundle $ E \rightarrow Q $ is an $\mathbb{R}$-bilinear
  map
  $ \nabla \colon \Gamma (A) \times \Gamma (E) \rightarrow \Gamma (E)
  $, $ ( X , s ) \mapsto \nabla _X s $, that is
  $ C ^\infty (Q) $-linear in the first argument and satisfies a
  Leibniz rule in the second, i.e.,
  \begin{equation*}
    \nabla _{ f X } s = f \nabla _X s , \qquad \nabla _X ( f s ) = \rho (X) [f] s + f \nabla _X s ,
  \end{equation*}
  for all $ f \in C ^\infty (Q) $.
\end{definition}

An $A$-connection on a vector bundle naturally induces an
$A$-connection on the dual bundle.

\begin{definition}
  Given an $A$-connection $ \nabla $ on $E$, the \emph{dual
    connection} $ \nabla ^\ast $ on $ E ^\ast $ is given by
  \begin{equation*}
    \langle \nabla ^\ast _X \sigma , s \rangle = \rho (X) \bigl[ \langle \sigma, s \rangle \bigr] - \langle \sigma, \nabla _X s \rangle ,
  \end{equation*}
  where $ X \in \Gamma (A) $, $ \sigma \in \Gamma ( E ^\ast ) $, and
  $ s \in \Gamma (E) $.
\end{definition}

For example, a $ T Q $-connection is the usual notion of a connection
on a vector bundle. Given a $ T Q $-connection $ \nabla $ on $A$,
there are two induced $A$-connections on $A$, which are denoted by
$ \nabla $ and $ \overline{ \nabla } $:
\begin{equation*}
  \nabla _X Y \coloneqq \nabla _{ \rho (X) } Y , \qquad \overline{ \nabla } _X Y \coloneqq \nabla _{ \rho (Y) } X + [X,Y].
\end{equation*}
In particular, if $ \nabla $ is the trivial $ T \bullet $-connection
on $ \mathfrak{g} \rightarrow \bullet $, then the induced
$\mathfrak{g}$-connections are
\begin{equation*}
  \nabla _\xi \eta = 0 , \qquad \overline{ \nabla } _\xi \eta = [ \xi, \eta ] = \operatorname{ad} _\xi \eta ,
\end{equation*}
for $ \xi, \eta \in \mathfrak{g} $. For $A$-connections on $A$ itself,
such as these, we have a notion of torsion.

\begin{definition}
  The \emph{torsion} of an $A$-connection $ \nabla $ on $A$ is defined
  to be
  \begin{equation*}
    T ( X, Y ) \coloneqq \nabla _X Y - \nabla _Y X - [X,Y] .
  \end{equation*}
\end{definition}
We note that
$ T \colon \Gamma (A) \times \Gamma (A) \rightarrow \Gamma (A) $ is
$ C ^\infty (Q) $-bilinear, since the Leibniz-rule terms from the
connection cancel with those from the bracket, so we may treat $T$ as
a tensor. Furthermore, the torsion of $ \nabla $ may be written
$ T ( X , Y ) = \nabla _X Y - \overline{ \nabla } _X Y $, which is
seen to be tensorial since it is the difference of two connections.

In addition to using an $A$-connection $ \nabla $ to differentiate
along sections, we will also use it to differentiate along paths.

\begin{definition}
  Suppose $ \nabla $ is an $A$-connection on $E$. Given an $A$-path
  $a$ over a base path $q$, choose a time-dependent section $\xi$ of
  $A$ such that $ \xi \bigl( t, q (t) \bigr) = a (t) $. Similarly,
  given a path $u$ in $E$ over the same base path $q$, choose a
  time-dependent section $\eta$ of $E$ such that
  $ \eta \bigl( t, q (t) \bigr) = u (t) $. We then define
  \begin{equation*}
    \nabla _a u (t) \coloneqq \partial _t \eta \bigl( t, q (t) \bigr) + \nabla _\xi \eta \bigl( t, q (t) \bigr) ,
  \end{equation*}
  which is independent of the choice of $\xi$ and $\eta$.
\end{definition}

Finally, we recall the notion of admissible variations of $A$-paths
from \citet{CrFe2003}, which can be readily expressed in terms of a
connection. As in \citep{CrFe2003}, let $ \widetilde{ P } (A) $ denote
the Banach manifold of $ C ^1 $ paths $I \rightarrow A$ having
$ C ^2 $ base paths $I \rightarrow Q$, and let
$ P (A) \subset \widetilde{ P } (A) $ denote the Banach submanifold of
$A$-paths.

\begin{definition}
  \label{def:a-path_variation}
  An \emph{admissible variation} of $a \in P (A) $ is a tangent vector
  $ X _{ b, a } \in T _a P (A) $, where $ b \in \widetilde{ P } (A) $
  covers the same base path and vanishes at the endpoints of $I$, such
  that the vertical and horizontal components relative to a
  $ T Q $-connection $ \nabla $ on $A$ are
  \begin{equation*}
    X _{ b, a } ^{\mathrm{ver}} = \overline{ \nabla } _a b , \qquad X _{ b, a } ^{\mathrm{hor}} = \rho (b).
  \end{equation*}
  By vertical and horizontal, we mean the components in the splitting
  $ T _a A \cong A _q \oplus T _q Q $ induced by $ \nabla $ for each
  $ t \in I $. Note that $ X _{ b , a } ^{\mathrm{hor}} \in T _q Q $
  is independent of the choice of connection.
\end{definition}

\begin{example}
  Recall that $ T Q $-paths are tangent prolongations of base paths
  $ q \colon I \rightarrow Q $. Since the anchor of $ T Q $ is the
  identity, admissible variations are completely determined by
  $ X _{ b , a } ^{\mathrm{hor}} = b $, which is an arbitrary
  variation of the base path (usually written $ \delta q $).
\end{example}

\begin{example}
  For a Lie algebra, recall that $\mathfrak{g}$-paths are arbitrary
  paths. Taking $ \nabla $ to be the trivial connection on
  $\mathfrak{g}$, we see that admissible variations of $\xi$ have the
  restricted form
  \begin{equation*}
    \delta \xi = \dot{ \eta } + \operatorname{ad} _\xi \eta ,
  \end{equation*}
  where $\eta$ is arbitrary.  In the context of Lagrangian mechanics
  on Lie algebras, these restrictions on admissible variations are
  called \emph{Lin constraints} (\citet[Chapter 13]{MaRa1999}).
\end{example}

\section{Lie--Poisson structure and equations relative to a
  connection}
\label{sec:lie-poisson}

\subsection{Splitting of \texorpdfstring{$ T ^\ast A ^\ast $}{T*A*}} As in the previous
section, let $ \tau \colon A \rightarrow Q $ be a Lie algebroid and
$ \pi \colon A ^\ast \rightarrow Q $ its dual. A $ TQ $-connection
$ \nabla $ on $A$ gives a splitting of $ T A $ into horizontal and
vertical subbundles, which induces a natural splitting of
$ T ^\ast A ^\ast $ as well.

The following lemma gives a convenient coordinate-free expression for
this splitting. This result seems like it ought to be standard, and it
may be known to experts, but we were unable to find it stated
explicitly in the literature. We remark that although the result is
stated for a Lie algebroid, which is our application of interest, it
holds for connections on vector bundles more generally.

\begin{lemma}
  \label{lem:splitting}
  A $ T Q $-connection $ \nabla $ on $A$ induces a splitting
  $ T ^\ast _p A ^\ast \cong A _q \oplus T ^\ast _q Q $, where
  $ p \in A ^\ast _q $. This splitting is characterized by the
  following condition: For all $ X \in \Gamma (A) $ and
  $ f \in C ^\infty (Q) $,
  \begin{subequations}
    \label{e:splitting}
    \begin{align}
      \mathrm{d} ( \Phi _X + f \circ \pi ) ^{\mathrm{ver}} (p)
      &= X (q) ,\label{e:splitting_ver}\\
      \mathrm{d} ( \Phi _X + f \circ \pi ) ^{\mathrm{hor}} (p)
      &= \langle p, \nabla X \rangle + \mathrm{d} f (q) \label{e:splitting_hor}.
    \end{align}
  \end{subequations}
  By $ \langle p, \nabla X \rangle \in T ^\ast _q Q $, we mean the
  covector whose pairing with $ v \in T _q Q $ is
  $ \langle p , \nabla _v X \rangle $.
\end{lemma}

\begin{proof}
  Given $ p \in A _q ^\ast $, recall that the vertical lift
  $ V _p \colon A _q ^\ast \rightarrow T _p A ^\ast $ is defined by
  \begin{equation*}
    V _p r \coloneqq \frac{\mathrm{d}}{\mathrm{d}\epsilon} ( p + \epsilon r ) \Bigr\rvert _{ \epsilon = 0 } .
  \end{equation*}
  Associated to the dual connection $ \nabla ^\ast $ on $ A ^\ast $,
  the horizontal lift $ H _p \colon T _q Q \rightarrow T _p A ^\ast $
  is given by
  \begin{equation*}
    H _p v = T \mu (v) - V _p ( \nabla _v ^\ast \mu ) ,
  \end{equation*}
  where $\mu \in \Gamma ( A ^\ast ) $ is any section satisfying
  $ \mu (q) = p $. This is independent of the choice of section, and
  together the vertical and horizontal lifts define the splitting
  \begin{equation*}
    V _p \oplus H _p \colon A _q ^\ast \oplus T _q Q \xrightarrow{ \cong } T _p A ^\ast .
  \end{equation*}
  See, for instance, \citet[\S3.3]{Wendl2008} and
  \citet*[\S17.9]{KoMiSl1993}.  By taking the dual of these vertical
  and horizontal lifts, we obtain projections
  \begin{alignat*}{2}
    V _p ^\ast &\colon T _p ^\ast A ^\ast \rightarrow A _q , &\qquad \alpha &\mapsto \alpha ^{\mathrm{ver}} ,\\
    H _p ^\ast &\colon T _p ^\ast A ^\ast \rightarrow T _q ^\ast Q , &\qquad \alpha &\mapsto \alpha ^{\mathrm{hor}} ,
  \end{alignat*}
  and thus a dual splitting
  \begin{equation*}
    V _p ^\ast \oplus H _p ^\ast \colon T ^\ast _p A ^\ast \xrightarrow{ \cong } A _q \oplus T ^\ast _q Q .
  \end{equation*}
  It remains to show that $ \alpha ^{\mathrm{ver}} $ and
  $ \alpha ^{\mathrm{hor}} $ have the claimed expressions for
  $ \alpha = \mathrm{d} ( \Phi _X + f \circ \pi ) (p) $.

  First, observe that for all $ r \in A _q ^\ast $, we have
  \begin{equation*}
    \langle \mathrm{d} \Phi _X , V _p r \rangle = \frac{\mathrm{d}}{\mathrm{d}\epsilon} \langle p + \epsilon r , X \rangle \Bigr\rvert _{ \epsilon = 0 } = \langle r, X \rangle .
  \end{equation*}
  Furthermore,
  \begin{equation*}
    \bigl\langle \mathrm{d} ( f \circ \pi ) , V _p r \bigr\rangle = \frac{\mathrm{d}}{\mathrm{d} \epsilon} ( f \circ \pi ) ( p + \epsilon r ) \Bigr\rvert _{ \epsilon = 0 } = 0 ,
  \end{equation*}
  since $ \pi ( p + \epsilon r ) = q $ is constant in
  $\epsilon$. Together, these establish \eqref{e:splitting_ver}. Next,
  for all $ v \in T _q Q $,
  \begin{align*}
    \langle \mathrm{d} \Phi _X , H _p v \rangle
    &= \bigl\langle \mathrm{d} \Phi _X , T \mu (v) - V _p ( \nabla ^\ast _v \mu ) \bigr\rangle \\
    &= v \bigl[ \langle \mu, X \rangle \bigr] - \langle \nabla _v ^\ast \mu, X \rangle \\
    &= \langle \mu, \nabla _v X \rangle ,
  \end{align*}
  where the last equality is the defining property of the dual
  connection $ \nabla ^\ast $. Finally,
  \begin{align*}
    \bigl\langle \mathrm{d} ( f \circ \pi ) , H _p v \bigr\rangle
    &= \bigl\langle \mathrm{d} ( f \circ \pi ) , T \mu (v) - V _p ( \nabla _v ^\ast \mu ) \bigr\rangle \\
    &= \bigl\langle \mathrm{d} ( f \circ \pi \circ \mu ) , v \bigr\rangle - 0 ,\\
    &= \langle \mathrm{d} f, v \rangle .
  \end{align*}
  Together, these establish \eqref{e:splitting_hor}, which completes the proof.
\end{proof}

\subsection{Lie--Poisson structure}

We now use the splitting in \cref{lem:splitting} to express the
Lie--Poisson structure on $ A ^\ast $ in terms of the connection
$ \nabla $. The following result is not new; we first learned of it
from Rui Loja Fernandes, and it appears as an exercise in
\citet*[Exercise 13.73]{CrFeMu2021}. However, the suggested proof in
\citep{CrFeMu2021} compares with a local coordinate expression for the
Poisson tensor, and we believe the following coordinate-free proof is
new.

\begin{lemma}
  \label{lem:poisson_tensor}
  Let $ \nabla $ be a $ T Q $-connection on $A$, and let
  $ T \colon A \otimes A \rightarrow A $ be the torsion tensor of the
  associated $A$-connection.  The $ ( \pm ) $ Lie--Poisson tensor
  $ \Pi _{ \pm } \colon T ^\ast A ^\ast \otimes T ^\ast A ^\ast
  \rightarrow \mathbb{R} $ satisfies
  \begin{equation*}
    \Pi _{ \pm } ( \alpha , \beta ) = \pm \Bigl[ \bigl\langle  \beta ^{\mathrm{hor}} , \rho ( \alpha ^{\mathrm{ver}}) \bigr\rangle - \bigl\langle \alpha ^{\mathrm{hor}} , \rho (\beta ^{\mathrm{ver}}) \bigr\rangle - \bigl\langle p, T ( \alpha ^{\mathrm{ver}} , \beta ^{\mathrm{ver}} ) \bigr\rangle  \Bigr] ,
  \end{equation*}
  for all $ \alpha , \beta \in T ^\ast _p A ^\ast $.
\end{lemma}

\begin{proof}
  It suffices to show that
  $ \{ F, G \} _\pm = \Pi _\pm ( \mathrm{d} F , \mathrm{d} G ) $
  satisfies the bracket relations
  \eqref{e:algebroid_bracket_relations}. First, if
  $ \alpha = \mathrm{d} \Phi _X (p) $ and
  $ \beta = \mathrm{d} \Phi _Y (p) $, then \cref{lem:splitting} gives
  \begin{align*}
    \Pi _\pm \bigl( \mathrm{d} \Phi _X (p) , \mathrm{d} \Phi _Y (p) \bigr)
    &= \pm \Bigl[ \langle p, \nabla _X Y \rangle - \langle p, \nabla _Y X \rangle - \bigl\langle p , T ( X, Y) \bigr\rangle \Bigr] \\
    &= \pm \bigl\langle p, [X, Y] \bigr\rangle ,
  \end{align*}
  verifying \eqref{e:algebroid_bracket_XY}. Next, if
  $ \alpha = \mathrm{d} \Phi _X (p) $ and
  $ \beta = \mathrm{d} ( g \circ \pi ) (p) $, then the two terms
  involving $ \beta ^{\mathrm{ver}} = 0 $ vanish, leaving only
  \begin{align*}
    \Pi _\pm \bigl( \mathrm{d} \Phi _X (p) , \mathrm{d} (g \circ \pi) (p) \bigr)
    &= \pm \Bigl\langle \mathrm{d} g (q) , \rho \bigl( X (q) \bigr)  \Bigr\rangle \\
    &= \pm \rho (X) [g] (q) ,
  \end{align*}  
  verifying \eqref{e:algebroid_bracket_Xg}. Finally, if
  $ \alpha = \mathrm{d} ( f \circ \pi ) (p) $ and
  $ \beta = \mathrm{d} ( g \circ \pi ) (p) $, then all three terms
  vanish, since $ \alpha ^{\mathrm{ver}} = \beta ^{\mathrm{ver}} = 0 $. Therefore,
  \begin{equation*}
    \Pi _\pm \bigl( \mathrm{d} ( f \circ \pi ) (p) , \mathrm{d} ( g \circ \pi ) (p) \bigr) = 0 ,
  \end{equation*}
  verifying \eqref{e:algebroid_bracket_fg} and completing the proof.
\end{proof}

\subsection{Lie--Poisson equations} We are now ready to state our main
result, expressing the $ ( \pm ) $ Lie--Poisson equations on
$ A ^\ast $ relative to a given $ T Q $-connection $ \nabla $ on $A$.

\begin{theorem}
  \label{thm:lie-poisson_connection}
  Given $ H \in C ^\infty ( A ^\ast ) $, a path $p$ in $ A ^\ast $
  with base path $q$ is an integral curve of the Hamiltonian vector
  field, with respect to the $ ( \pm ) $ Lie--Poisson structure, if
  and only if
  \begin{subequations}
    \label{e:lie-poisson_connection}
    \begin{align}
      \dot{q} &= \rho (a) ,\label{e:lie-poisson_connection_hor}\\ \overline{ \nabla } ^\ast _a p &= \pm \rho ^\ast \bigl( \mathrm{d} H ^{\mathrm{hor}} (p) \bigr), \label{e:lie-poisson_connection_ver}
    \end{align}
  \end{subequations}
  where $ a = \mp \mathrm{d} H ^{\mathrm{ver}} (p) $. The first
  equation says that $a$ is an $A$-path, so it is valid to write
  $ \overline{ \nabla } _a ^\ast $.
\end{theorem}

\begin{proof}
  By definition, $p$ is an integral curve of $H$ if and only if
  $ \frac{\mathrm{d}}{\mathrm{d}t} F (p) = \{ F, H \} _\pm (p) $ for
  all $ F \in C ^\infty ( A ^\ast ) $. It suffices to consider
  $ F = f \circ \pi $ for $ f \in C ^\infty (Q) $ and $ F = \Phi _X $
  for $ X \in \Gamma (A) $.

  First, by the chain rule,
  \begin{equation*}
    \frac{\mathrm{d}}{\mathrm{d}t} ( f \circ \pi ) (p) = \bigl\langle \mathrm{d} f (q) , \dot{q} \bigr\rangle ,
  \end{equation*}
  while \cref{lem:splitting,lem:poisson_tensor} imply
  \begin{equation*}
    \{ f \circ \pi , H \} _\pm (p) = \mp \Bigl\langle \mathrm{d} f (q), \rho \bigl( \mathrm{d} H ^{\mathrm{ver}} (p) \bigr) \Bigr\rangle = \bigl\langle \mathrm{d} f (q), \rho (a) \bigr\rangle .
  \end{equation*}
  These are equal for all $ f \in C ^\infty (Q) $ if and only if
  \eqref{e:lie-poisson_connection_hor} holds, i.e.,
  $a = \mp \mathrm{d} H ^{\mathrm{ver}} (p) $ is an $A$-path.

  Next, if $\mu$ is a time-dependent section of $ A ^\ast $ such that
  $ \mu \bigl( t , q (t) \bigr) = p (t) $, then
  \begin{align*}
    \frac{\mathrm{d}}{\mathrm{d}t} \Phi _X (p)
    &= \frac{\mathrm{d}}{\mathrm{d}t} \langle \mu , X \rangle (q) \\
    &= \langle \partial _t \mu , X \rangle + \rho (a)  \bigl[ \langle \mu , X \rangle \bigr] \\
    &= \langle \partial _t \mu + \overline{ \nabla } ^\ast _a \mu , X \rangle + \langle \mu , \overline{ \nabla } _a X \rangle \\
    &= \langle \overline{ \nabla } _a ^\ast p , X \rangle + \langle p, \overline{ \nabla } _a X \rangle ,
  \end{align*}
  while \cref{lem:splitting,lem:poisson_tensor} imply
  \begin{align*}
    \{ \Phi _X , H \} _\pm (p) 
    &= \pm \biggl[ \bigl\langle \mathrm{d} H ^{\mathrm{hor}} (p) , \rho (X) \bigr\rangle - \langle p, \nabla _{ \mathrm{d} H ^{\mathrm{ver}} (p) } X \rangle - \Bigl\langle p , T \bigl(  X, \mathrm{d} H ^{\mathrm{ver}} (p) \bigr) \Bigr\rangle \biggr]\\
    &= \pm \Bigl[ \bigl\langle \mathrm{d} H ^{\mathrm{hor}} (p) , \rho (X) \bigr\rangle - \langle p , \overline{ \nabla } _{ \mathrm{d} H ^{\mathrm{ver}} (p) } X \rangle \Bigr] \\
    &= \pm \Bigl\langle \rho ^\ast \bigl( \mathrm{d} H ^{\mathrm{hor}} (p) \bigr) , X \Bigr\rangle + \langle p, \overline{ \nabla } _a X \rangle ,
  \end{align*}
  using the fact that torsion is the difference between $ \nabla $ and
  $ \overline{ \nabla } $. These two expressions are equal for all
  $ X \in \Gamma (A) $ if and only if
  \eqref{e:lie-poisson_connection_ver} holds.
\end{proof}

We now relate the coordinate-free form of the Lie--Poisson equations
in \cref{thm:lie-poisson_connection} to the local-coordinate
formulation \eqref{e:lie-poisson_coordinates}. Given local coordinates
$ q ^i $ for $Q$ and a local basis of sections $ \{ e _I \} $ of $A$,
choose the locally trivial $ T Q $-connection defined by
$ \nabla _{ \partial / \partial q ^i } e _J = 0 $. With respect to
this connection, the vertical and horizontal components of
$ \mathrm{d} H $ are given by
\begin{equation*}
  \mathrm{d} H ^{\mathrm{ver}} = \frac{ \partial H }{ \partial \mu _I } e _I , \qquad \mathrm{d} H ^{\mathrm{hor}} = \frac{ \partial H }{ \partial q ^i } \,\mathrm{d}q ^i .
\end{equation*}
It follows that $ a = \mp \mathrm{d} H ^{\mathrm{ver}} (p) $ is an
$A$-path if and only if
\begin{equation*}
  \dot{q} ^i = \mp \rho ^i _I \frac{ \partial H }{ \partial \mu _I } ,
\end{equation*}
which is \eqref{e:lie-poisson_coordinates_q}. Next, if $\mu$ is a
time-dependent section of $ A ^\ast $ such that
$ \mu \bigl( t, q (t) \bigr) = p (t) $, then
\begin{equation*}
  \langle \overline{ \nabla } _a ^\ast p , e _I \rangle = \langle \partial _t \mu + \overline{ \nabla } _a ^\ast \mu , e _I \rangle = \dot{ \mu } _I \mp C _{ I J } ^K \frac{ \partial H }{ \partial \mu _J } \mu _K .
\end{equation*}
Setting this equal to
$ \pm \bigl\langle \mathrm{d} H ^{\mathrm{hor}} (p), \rho ( e _I )
\bigr\rangle $ gives
\begin{equation*}
  \dot{ \mu } _I \mp C _{ I J } ^K \frac{ \partial H }{ \partial \mu _J } \mu _K = \pm \rho ^i _I\frac{ \partial H }{ \partial q ^i } ,
\end{equation*}
which rearranges to \eqref{e:lie-poisson_coordinates_mu}.

\begin{remark}
  \label{rmk:christoffel}
  To illustrate that these equations are independent of the
  $ T Q $-connection chosen, suppose more generally that
  $ \nabla _{ \partial / \partial q ^i } e _J = \Gamma _{ i J } ^K e
  _K $, where $ \Gamma _{ i J } ^K $ are Christoffel symbols. Then
  \begin{equation*}
    \mathrm{d} H ^{\mathrm{ver}} = \frac{ \partial H }{ \partial \mu _I } e _I , \qquad \mathrm{d} H ^{\mathrm{hor}} = \biggl( \frac{ \partial H }{ \partial q ^i } + \Gamma _{ i J } ^K \frac{ \partial H }{ \partial \mu _J } \mu _K \biggr) \,\mathrm{d} q ^i .
  \end{equation*}
  Since $ \mathrm{d} H ^{\mathrm{ver}} $ is the same as above, once
  again \eqref{e:lie-poisson_connection_hor} becomes
  \eqref{e:lie-poisson_coordinates_q}. Meanwhile,
  \eqref{e:lie-poisson_connection_ver} becomes
  \begin{equation*}
    \dot{ \mu } _I \pm ( \rho _I ^i \Gamma _{ i J } ^K - C _{ I J } ^K ) \frac{ \partial H }{ \partial \mu _J } \mu _K = \pm \rho ^i _I \biggl( \frac{ \partial H }{ \partial q ^i } + \Gamma _{ i J } ^K \frac{ \partial H }{ \partial \mu _J } \mu _K \biggr) ,
  \end{equation*}
  and the connection-dependent terms cancel to give
  \eqref{e:lie-poisson_coordinates_mu}.
\end{remark}

\subsection{Examples}

We now illustrate the results of this section by showing how they give
the canonical structures and dynamics on Lie algebra duals and
cotangent bundles.

\begin{example}
  If $ A = \mathfrak{g} \rightarrow \bullet $ is a Lie algebra and
  $ \nabla $ is the trivial $ T \bullet $-connection, then we have the
  associated $A$-connections $ \nabla = 0 $,
  $ \overline{ \nabla } = \operatorname{ad} $, and
  $ \overline{ \nabla } ^\ast = -\operatorname{ad} ^\ast $.

  Given $ F \in C ^\infty ( \mathfrak{g} ^\ast ) $, we have
  $ \mathrm{d} F ^{\mathrm{ver}} = \delta F / \delta \mu $ and
  $ \mathrm{d} F ^{\mathrm{hor}} = 0 $, and likewise for
  $ G \in C ^\infty ( \mathfrak{g} ^\ast ) $. Since the horizontal
  components vanish and $ T ( \cdot , \cdot ) = - [ \cdot , \cdot ] $,
  applying \cref{lem:poisson_tensor} gives
  \begin{equation*}
    \{ F, G \} _\pm (\mu) = \mp \Bigl\langle \mu , T \bigl( \mathrm{d} F ^{\mathrm{ver}} (\mu) , \mathrm{d} G ^{\mathrm{ver}} (\mu) \bigr) \Bigr\rangle = \pm \Biggl\langle \mu , \biggl[ \frac{ \delta F }{ \delta \mu } , \frac{ \delta G }{ \delta \mu } \biggr]\Biggr\rangle,
  \end{equation*}
  recovering the $ ( \pm ) $ Lie--Poisson brackets on
  $ \mathfrak{g}^{\ast} $ from \eqref{e:algebra_bracket}.

  Applying \cref{thm:lie-poisson_connection} to
  $ H \in C ^\infty ( \mathfrak{g} ^\ast ) $, the horizontal
  components vanish, and we have $ a = \mp \delta H / \delta \mu
  $. Identifying $p$ with the time-dependent section
  $ \mu ( t, \bullet ) = p (t) $, we get
  \begin{equation*}
    \dot{ \mu } \pm \operatorname{ad} ^\ast _{ \delta H / \delta \mu } \mu = \dot{ \mu } + \overline{ \nabla } ^\ast _a \mu = \overline{ \nabla } _a ^\ast p = 0 ,
  \end{equation*}
  so \eqref{e:lie-poisson_connection_hor} is trivial and
  \eqref{e:lie-poisson_connection_ver} is equivalent to the
  Lie--Poisson equations \eqref{e:lie-poisson} on
  $\mathfrak{g}^{\ast}$.
\end{example}

\begin{example}
  Suppose $ A = T Q \rightarrow Q $ is a tangent bundle. As in
  \cref{ex:tq_poisson}, let $ q ^i $ be local coordinates on $Q$,
  $ \partial / \partial q ^i $ be the corresponding local basis of
  sections of $ T Q $, and $ q ^i , p _i $ be canonical coordinates
  for $ T ^\ast Q $. Take the locally trivial connection
  $ \nabla _{ \partial / \partial q ^i } ( \partial / \partial q ^j )
  = 0 $. (This may not be globally trivial, but \cref{rmk:christoffel}
  shows that there is no loss of generality.) Since
  $ [ \partial / \partial q ^i , \partial / \partial q ^j ] = 0 $, it
  follows that $ \overline{ \nabla } $ and
  $ \overline{ \nabla } ^\ast $ are also trivial in local coordinates.

  Given $ F \in C ^\infty (T ^\ast Q ) $, the vertical and horizontal
  components of $ \mathrm{d} F $ are
  \begin{equation*}
    \mathrm{d} F ^{\mathrm{ver}} = \frac{ \partial F }{ \partial p _i } \frac{ \partial }{ \partial q ^i } , \qquad \mathrm{d} F ^{\mathrm{hor}} = \frac{ \partial F }{ \partial q ^i } \,\mathrm{d}q ^i ,
  \end{equation*}
  and likewise for $ G \in C ^\infty ( T ^\ast Q ) $. Since $\rho$ is
  the identity on $ T Q $ and $ \nabla $ is torsion-free,
  \cref{lem:poisson_tensor} gives
  \begin{equation*}
    \{ F, G \} _\pm
    = \pm \bigl[ \langle \mathrm{d} G ^{\mathrm{hor}} , \mathrm{d} F ^{\mathrm{ver}} \rangle - \langle \mathrm{d} F ^{\mathrm{hor}} , \mathrm{d} G ^{\mathrm{ver}} \rangle \bigr]
    = \pm \biggl( \frac{ \partial F }{ \partial p _i } \frac{ \partial G }{ \partial q ^i } - \frac{ \partial F }{ \partial q ^i } \frac{ \partial G }{ \partial p _i } \biggr),
  \end{equation*}
  recovering the previous expression \eqref{e:cotangent_bracket} for
  the $( \pm ) $ brackets. Again, we note that the $ ( - ) $ bracket
  agrees with the commonly-used sign convention for the canonical
  Poisson structure on a cotangent bundle, e.g., as in
  \citet{MaRa1999}.

  Applying \cref{thm:lie-poisson_connection} to
  $ H \in C ^\infty ( T ^\ast Q ) $, we have
  \begin{equation*}
    a = \mp \frac{ \partial H }{ \partial p _i } \frac{ \partial }{ \partial q ^i } .
  \end{equation*}
  Since $\rho$ is the identity and $ \overline{ \nabla } ^\ast $ is
  trivial in local coordinates, \eqref{e:lie-poisson_connection} gives
  \begin{align*}
    \dot{q} ^i &= \mp \frac{ \partial H }{ \partial p _i } ,\\
    \dot{p} _i &= \pm \frac{ \partial H }{ \partial q ^i } ,
  \end{align*}
  recovering the $ ( \pm ) $ Hamilton's equations
  \eqref{e:cotangent_equations}.
\end{example}

We give one more example, on the cotangent bundle, where we apply
\cref{thm:lie-poisson_connection} to a Hamiltonian arising from a
metric on $Q$, obtaining coordinate-free dynamics in terms of the
Levi-Civita connection.

\begin{example}
  \label{ex:geodesic}
  Let $g \colon T Q \otimes T Q \rightarrow \mathbb{R} $ be a
  (pseudo-)Riemannian metric on $Q$. Let
  $ g ^\flat \colon T Q \rightarrow T ^\ast Q $ denote the bundle map
  $ v \mapsto g ( v, \cdot ) $ and
  $ g ^\sharp \coloneqq ( g ^\flat ) ^{-1} \colon T ^\ast Q
  \rightarrow T Q $. Now, suppose we have a Hamiltonian in the form
  \begin{equation}
    \label{e:kinetic-plus-potential}
    H ( q, p ) = \frac{1}{2} \bigl\langle p, g ^\sharp (p) \bigr\rangle + U ( q ) ,
  \end{equation}
  where $ U \in C ^\infty (Q) $ is a potential energy function.

  Equip $ T ^\ast Q $ with the canonical $ ( - ) $ Poisson structure,
  and let $ \nabla $ be the Levi-Civita connection on $Q$. It follows that
  \begin{equation*}
    \mathrm{d} H ^{\mathrm{ver}} (q, p) = g ^\sharp (p) , \qquad \mathrm{d} H ^{\mathrm{hor}} (q, p ) = \mathrm{d} U ( q ) ,
  \end{equation*}
  where the kinetic-energy part of the Hamiltonian does not contribute
  to the horizontal part due to the metric-compatibility of
  $ \nabla $. Furthermore, since $ \nabla $ is torsion-free, we have
  $ \nabla = \overline{ \nabla } $. Since $\rho$ is the identity and
  $ a = g ^\sharp (p) $, we see that \eqref{e:lie-poisson_connection} gives the equations
  \begin{align*}
    \dot{q} &= g ^\sharp (p) ,\\
    \nabla _{ \dot{q} } ^\ast p &= -\mathrm{d} U (q) .
  \end{align*}
  In particular, the case $ U = 0 $ corresponds to geodesic flow. Note that applying $ g ^\sharp $ to both sides of the second equation, and substituting the first, gives the second-order dynamics
  \begin{equation*}
    \nabla _{ \dot{q} } \dot{q} = - \operatorname{grad} U (q) .
  \end{equation*}
\end{example}

\subsection{Systems arising from Lie algebroid metrics}

We now generalize \cref{ex:geodesic} to the case where $A$ is an
arbitrary Lie algebroid over $Q$, equipped with a bundle metric
$ g \colon A \otimes A \rightarrow \mathbb{R} $, and
\eqref{e:kinetic-plus-potential} is a kinetic-plus-potential
Hamiltonian on $ A ^\ast $. Our approach extends that of
\citet{Martinez2001} for the case of action algebroids; see also
\citet{CoMa2004} for a related approach in the context of Lagrangian
control systems.

Some caution is required: while there is a suitable notion of
Levi-Civita (metric-compatible, torsion-free) $A$-connection
\citep{CoMa2004,CoLeMaMaMa2006,GrUrGr2006}, this $A$-connection
generally \emph{does not} arise from a $ T Q $-connection $ \nabla
$. For example, when $ A = \mathfrak{g} \rightarrow \bullet $, the
trivial connection is the unique $ T \bullet $-connection, but its
associated $A$-connection has torsion $ - [\cdot , \cdot ]$, which is
nonvanishing unless $\mathfrak{g}$ is abelian. Instead, since $g$ is a
bundle metric, we may only assume the existence (but not necessarily
uniqueness) of a metric-compatible $ T Q $-connection $ \nabla $,
whose associated $A$-connection may have nonvanishing torsion. (See
the remark following Proposition~III.1.5 in \citet{KoNo1963} on the
existence of bundle-metric-compatible connections.)

Metric-compatibility of $ \nabla $ gives the same splitting of
$ \mathrm{d} H $ into vertical and horizontal parts as in
\cref{ex:geodesic}. Taking the $ ( - ) $ Lie--Poisson structure, we
therefore have $ a = g ^\sharp (p) $ and
\begin{subequations}
  \label{e:algebroid_newton}
  \begin{align}
    \dot{q} &= \rho (a) , \label{e:algebroid_newton_q}\\
    \overline{ \nabla } _a ^\ast p &= - \rho ^\ast \bigl( \mathrm{d} U ( q ) \bigr) \label{e:algebroid_newton_p}.
  \end{align}
\end{subequations}
Here, $ \overline{ \nabla } $ generally differs from $ \nabla $ and
need not be metric-compatible, so applying $ g ^\sharp $ to
\eqref{e:algebroid_newton_p} is not as simple as before. We first
define an additional $A$-connection
\begin{equation*}
  \overline{ \nabla } ^\dagger _X Y \coloneqq g ^\sharp \bigl( \nabla ^\ast _X g ^\flat ( Y ) \bigr) , \qquad X, Y \in \Gamma (A) ,
\end{equation*}
which is directly verified to satisfy \cref{def:a-connection}.  Next,
as in \citet{CoMa2004}, we define the gradient of $U$ with respect to
the Lie algebroid metric to be the section
\begin{equation*}
  \operatorname{grad} U \coloneqq g ^\sharp \circ \rho ^\ast \circ \mathrm{d} U \in \Gamma (A) .
\end{equation*}
Therefore, applying $ g ^\sharp $ to both sides of
\eqref{e:algebroid_newton_p} gives
\begin{equation}
  \label{e:algebroid_newton_a}
  \overline{ \nabla } _a ^\dagger a = - \operatorname{grad} U (q) ,
\end{equation}
and the case $ U = 0 $ gives geodesic flow with respect to the
$A$-connection $ \overline{ \nabla } ^\dagger $.

Although $ \overline{ \nabla } ^\dagger $ is generally distinct from
the Levi-Civita $A$-connection $ \nabla ^g $, we next show that in
fact $ \overline{ \nabla } ^\dagger _a a = \nabla ^g _a a $. Thus
$ \overline{ \nabla } ^\dagger $ and $ \nabla ^g $ are interchangeable
in \eqref{e:algebroid_newton_a}, and in particular they have the same
geodesic flow.

\begin{lemma}
  \label{lem:levi-civita}
  Let $A$ be a Lie algebroid over $Q$ equipped with a bundle metric
  $g$. For any bundle-metric-compatible $ T Q $-connection $ \nabla $,
  the Levi-Civita connection $ \nabla ^g $ satisfies
  \begin{equation*}
    \nabla ^g _X Y = \frac{1}{2} \bigl( \overline{ \nabla } ^\dagger _X Y + \overline{ \nabla } ^\dagger _Y X + [ X, Y ] \bigr) ,
  \end{equation*}
  for all $ X, Y \in \Gamma (A) $. That is,
  $ \nabla ^g = \frac{1}{2} \bigl( \overline{ \nabla } ^\dagger +
  \overline{ \overline{ \nabla } ^\dagger } \bigr) $.
\end{lemma}

\begin{proof}
  From \citep[Proposition 2.6]{CoMa2004}, the Levi-Civita connection
  is determined by
  \begin{alignat*}{4}
    2 g ( \nabla ^g _X Y, Z )
    &={}& \rho (X) \bigl[ g ( Y, Z ) \bigr] &{}+{}& \rho (Y) \bigl[ g ( X, Z ) \bigr] &{}-{}& \rho (Z) \bigl[ g ( X, Y ) \bigr] \\
    &&{}-  g \bigl( X, [Y,Z] \bigr) &{}-{}& g \bigl( Y, [X,Z] \bigr) &{}+{}& g \bigl( Z, [X,Y] \bigr) ,
  \end{alignat*} 
  for all $ X, Y, Z \in \Gamma (A) $, which is a generalization of the
  usual tangent bundle formula. Now, by the definitions of
  $ \overline{ \nabla } ^\dagger $, the dual connection
  $ \overline{ \nabla } ^\ast $, and $ \overline{ \nabla } $ itself,
  we have
  \begin{align*}
    \rho (X) \bigl[ g ( Y,Z ) \bigr]
    &= g ( \overline{ \nabla }^\dagger _X Y , Z ) + g ( Y, \overline{ \nabla } _X Z ) \\
    &= g ( \overline{ \nabla }^\dagger _X Y , Z ) + g ( Y, \nabla _Z X + [X,Z] ),
  \end{align*} 
  and likewise,
  \begin{equation*}
    \rho (Y) \bigl[ g ( X,Z ) \bigr]
    = g ( \overline{ \nabla }^\dagger _Y X , Z ) + g ( X, \nabla _Z Y + [Y,Z] )   \end{equation*}
  On the other hand, metric-compatibility of $ \nabla $ gives
  \begin{equation*}
    \rho (Z) \bigl[ g (X,Y) \bigr] = g ( \nabla _Z X , Y ) + g ( X , \nabla _Z Y ) .
  \end{equation*}
  Substituting these into the formula for $ \nabla ^g $ and canceling
  terms, we obtain
  \begin{equation*}
    2 g ( \nabla ^g _X Y , Z ) = g \bigl( \overline{ \nabla } ^\dagger _X Y + \overline{ \nabla } ^\dagger _Y X + [X,Y], Z \bigr) ,
  \end{equation*}
  which completes the proof.
\end{proof}

\begin{corollary}
  For all $ X \in \Gamma (A) $, we have
  $ \nabla ^g _X X = \overline{ \nabla } ^\dagger _X X $. In
  particular, an $A$-path $a$ satisfies \eqref{e:algebroid_newton_a}
  if and only if
  \begin{equation*}
    \nabla ^g _a a = - \operatorname{grad} U (q) .
  \end{equation*}
\end{corollary}

\begin{example}
  When $ \nabla $ is the Levi-Civita connection on
  $ A = T Q \rightarrow Q $, we have
  $ \overline{ \nabla } ^\dagger = \overline{ \nabla } = \nabla $ and
  thus recover the equations obtained in \cref{ex:geodesic}.
\end{example}

\begin{example}
  The special case where $A$ is an action algebroid recovers the
  equations of \citet{Martinez2001}, as we now show. An
  (infinitesimal) action of a Lie algebra $\mathfrak{g}$ on $Q$ is a
  Lie algebra homomorphism
  $ \mathfrak{g} \rightarrow \mathfrak{X} (Q) $,
  $ \xi \mapsto \xi _Q $. The associated \emph{action algebroid} is
  the trivial bundle $ A = Q \times \mathfrak{g} \rightarrow Q $,
  where $ \rho ( q, \xi ) = \xi _Q (q) $. The Lie bracket on
  $ \Gamma (A) $ is determined by requiring that it agree with the
  bracket of $\mathfrak{g}$ on constant sections, where
  $ \xi \in \mathfrak{g} $ is identified with the constant section
  $ q \mapsto ( q, \xi ) $, and extended to arbitrary sections by the
  Leibniz rule. We take the standard connection on a trivial bundle,
  where $ \nabla \xi = 0 $ on constant sections.

  Now, an inner product on $\mathfrak{g}$ gives a bundle metric on $A$
  that is constant with respect to the basepoint, i.e.,
  $ g _q ( \xi, \eta ) = g ( \xi, \eta ) $, so $ \nabla $ is a
  metric-compatible connection. It follows that
  $ \overline{ \nabla } _\xi \eta = \operatorname{ad} _\xi \eta $,
  $ \overline{ \nabla } ^\ast _\xi \mu = -\operatorname{ad} ^\ast _\xi
  \mu $, and
  $ \overline{ \nabla } ^\dagger _\xi \eta = - \operatorname{ad}
  ^\dagger _\xi \eta $ for constant sections
  $ \xi, \eta \in \mathfrak{g} $ and $ \mu \in \mathfrak{g}^{\ast}
  $. Writing $ a (t) = \bigl( q (t) , \xi (t) \bigr) $ and
  $ p (t) = \bigl( q (t) , \mu (t) \bigr) $, we have
  $ \xi = g ^\sharp (\mu) $, and \eqref{e:algebroid_newton} becomes
  \begin{align*}
    \dot{q} &= \xi _Q (q) ,\\
    \dot{ \mu } - \operatorname{ad} _\xi ^\ast \mu &= - \rho ^\ast \bigl( \mathrm{d} U (q) \bigr) .
  \end{align*}
  Applying $ g ^\sharp $ to both sides of the second equation gives
  \begin{equation*}
    \dot{ \xi } - \operatorname{ad} _\xi ^\dagger \xi = - \operatorname{grad} U (q) ,
  \end{equation*}
  as in \citet{Martinez2001}. \cref{lem:levi-civita} recovers the
  formula
  $ \nabla ^g _\xi \eta = \frac{1}{2} \bigl( - \operatorname{ad}
  ^\dagger _\xi \eta - \operatorname{ad} ^\dagger _\eta \xi + [\xi,
  \eta ] \bigr) $ for constant sections
  $ \xi, \eta \in \mathfrak{g} $, as in \citet[Section 7]{CoMa2004}.
\end{example}

\begin{remark}
  See \citet*{GrUrGr2006} for an application of generalized geodesic
  flow on a Lie algebroid to Wong's equations \citep{Wong1970}, which
  are a classic example in reduction theory. See also \citet*[Example
  3.18]{LiStTa2017} for a discussion of this example from the
  Lagrangian point of view.
\end{remark}

\section{Generalization of Hamilton's phase space principle}
\label{sec:variational}

\subsection{The variational principle} In this section, we establish a
variational principle for a Hamiltonian
$ H \in C ^\infty ( A ^\ast ) $ whose critical paths are solutions to
the generalized Lie--Poisson dynamics
\eqref{e:lie-poisson_connection}. We begin by describing the
admissible paths and variations.

\begin{definition}
  \label{def:a+a*}
  An \emph{$ ( A \oplus A ^\ast ) $-path} is a $ C ^1 $ path
  $ ( a, p ) \colon I \rightarrow A \oplus A ^\ast $ over a $ C ^2 $
  base path $ q = \tau \circ a = \pi \circ p \colon I \rightarrow Q $
  such that $ \dot{q} = \rho (a) $, i.e., $a$ is an $A$-path. Let
  $ P ( A \oplus A ^\ast ) \subset \widetilde{ P } ( A \oplus A ^\ast
  ) $ denote the Banach submanifold of $ ( A \oplus A ^\ast ) $-paths,
  among all $ C ^1 $ paths with $ C ^2 $ base paths.

  An \emph{admissible variation} of
  $ (a,p) \in P ( A \oplus A ^\ast ) $ is a tangent vector
  $ ( \delta a, \delta p ) \in T _{ ( a,p ) } P ( A \oplus A ^\ast ) $
  such that $ \delta a = X _{ b, a } \in T _a P ( A ) $ is admissible
  in the sense of \cref{def:a-path_variation}. Relative to a
  $ T Q $-connection $ \nabla $ on $A$, we have
  \begin{equation*}
    \delta a ^{\mathrm{ver}} = \overline{ \nabla } _a b , \qquad \delta p ^{\mathrm{ver}} = r , \qquad \delta a ^{\mathrm{hor}} = \delta p ^{\mathrm{hor}} = \rho (b) ,
  \end{equation*}
  where $ ( b, r ) \in \widetilde{ P } ( A \oplus A ^\ast ) $ covers
  the same base path $q$, where $b$ vanishes at the endpoints of $I$,
  and where $b$ and $r$ may otherwise be arbitrary.
\end{definition}

Note that $ \delta a ^{\mathrm{hor}} = \delta p ^{\mathrm{hor}} $ is
necessary for $ ( \delta a, \delta p ) $ to be tangent to
$ P ( A \oplus A ^\ast ) $, which we can see by differentiating the
condition $ \tau \circ a = \pi \circ p $.

\begin{remark}
  \label{rmk:a+a*_algebroid}
  An alternative perspective is that
  $ A \oplus A ^\ast \rightarrow Q $ is in fact a Lie algebroid, where
  the anchor is $ ( a, p ) \mapsto \rho (a) $ and the Lie bracket is
  just that for the $A$-component. From this point of view,
  $ ( A \oplus A ^\ast ) $-paths and admissible variations are a
  special case of the earlier definitions in \cref{sec:algebroids}. In
  particular, we may make use of the results in \citet[Section
  4.2]{CrFe2003} regarding the Banach manifold structure of
  $ P ( A \oplus A ^\ast ) $.
\end{remark}

The usual techniques of calculus of variations may now be applied to
$ P ( A \oplus A ^\ast ) $. Given a functional
$ S \colon P ( A \oplus A ^\ast ) \rightarrow \mathbb{R} $, we denote
\begin{equation*}
  \delta S ( a, p ) \coloneqq \bigl\langle \mathrm{d} S (a,p) , ( \delta a,
  \delta p ) \bigr\rangle ,
\end{equation*}
where $ ( \delta a, \delta p ) $ is an admissible variation of
$ (a, p) $. If
$ \epsilon \mapsto ( a _\epsilon, p _\epsilon ) \in P ( A \oplus A
^\ast ) $ is a curve, i.e., a homotopy of
$ ( A \oplus A ^\ast ) $-paths, such that
$ \frac{\mathrm{d}}{\mathrm{d}\epsilon } ( a _\epsilon, p _\epsilon )
\bigr\rvert _{ \epsilon = 0 } = ( \delta a, \delta p ) $, then we have
$ \delta S ( a, p ) = \frac{\mathrm{d}}{\mathrm{d}\epsilon} S ( a
_\epsilon, p _\epsilon ) \bigr\rvert _{ \epsilon = 0 } $. In
particular, when $S$ has the form
\begin{equation*}
  S ( a, p ) = \int _I \mathcal{L} \bigl( a (t) , p (t) \bigr) \,\mathrm{d}t ,
\end{equation*}
for some
$ \mathcal{L} \colon A \oplus A ^\ast \rightarrow \mathbb{R} $, then
differentiating under the integral sign gives
\begin{equation*}
  \delta S ( a, p ) = \int _I \frac{ \partial }{ \partial \epsilon } \mathcal{L} \bigl( a _\epsilon (t) , p _\epsilon (t) \bigr) \Bigr\rvert _{ \epsilon = 0 } \,\mathrm{d}t = \int _I \Bigl\langle \mathrm{d} \mathcal{L} \bigl( a (t), p (t) \bigr) , \bigl( \delta a (t) , \delta p (t) \bigr) \Bigr\rangle \,\mathrm{d}t .
\end{equation*}
For notational simplicity, we now suppress the dependence of the
integrand on $t$. The $ T Q $-connection $ \nabla $ can be used to
expand the integrand into vertical and horizontal parts,
\begin{equation*}
  \bigl\langle \mathrm{d} \mathcal{L} ( a, p ) , ( \delta a , \delta p ) \bigr\rangle
  = \bigl\langle \mathrm{d} \mathcal{L}  ^{\mathrm{ver}} ( a, p ) , ( \overline{ \nabla } _a b , r ) \bigr\rangle + \bigl\langle \mathrm{d} \mathcal{L} ^{\mathrm{hor}} ( a, p ) , \rho (b) \bigr\rangle .
\end{equation*}
Denoting
$ \mathrm{d} \mathcal{L} ^{\mathrm{ver}} \eqqcolon ( \mathrm{d}
\mathcal{L} ^{\mathrm{ver}} _a , \mathrm{d} \mathcal{L}
^{\mathrm{ver}} _p ) $, we therefore obtain the variational formula
\begin{align*}
  \delta S ( a, p )
  &= \int _I \Bigl( \bigl\langle \mathrm{d} \mathcal{L} _a ^{\mathrm{ver}} ( a, p ), \overline{ \nabla } _a b \bigr\rangle + \bigl\langle r, \mathrm{d} \mathcal{L} _p ^{\mathrm{ver}} ( a, p ) \bigr\rangle + \bigl\langle \mathrm{d} \mathcal{L} ^{\mathrm{hor}} ( a, p ), \rho (b) \bigr\rangle \Bigr) \,\mathrm{d}t \\
  &= \int _I \biggl( \Bigl\langle - \overline{ \nabla } _a ^\ast \mathrm{d} \mathcal{L} ^{\mathrm{ver}} _a ( a, p ) + \rho ^\ast \bigl( \mathrm{d} \mathcal{L} ^{\mathrm{hor}} ( a, p ) \bigr) , b \Bigr\rangle + \bigl\langle r, \mathrm{d} \mathcal{L} _p ^{\mathrm{ver}} ( a, p ) \bigr\rangle \biggr) \,\mathrm{d}t 
\end{align*}
On the second line, we integrate by parts using the dual connection
$ \overline{ \nabla } _a ^\ast $ and the fact that $b$ vanishes at the
endpoints of $I$. Since $b$ and $r$ are otherwise arbitrary paths over
$q$, we conclude that $ \delta S ( a, p ) = 0 $ for all admissible
variations $ ( \delta a, \delta p ) $ if and only if
\begin{equation}
  \label{e:a+a*_elp}
\overline{ \nabla } _a ^\ast \mathrm{d} \mathcal{L} ^{\mathrm{ver}} _a ( a, p ) =   \rho ^\ast \bigl( \mathrm{d} \mathcal{L} ^{\mathrm{hor}} ( a, p ) \bigr), \qquad \mathrm{d} \mathcal{L} _p ^{\mathrm{ver}} ( a, p ) = 0 .
\end{equation}
In light of \cref{rmk:a+a*_algebroid}, this can be seen as a
particular case of the Euler--Lagrange--Poincar\'e equations in
\citet[Theorem 3.12]{LiStTa2017}, discussed further in
\cref{sec:lagrangian}, where $\mathcal{L}$ is a Lagrangian on
$ A \oplus A ^\ast $ viewed as a Lie algebroid.

Note that, using the vertical lifts
$ V _a \colon A _q \rightarrow T _a A $ and
$ V _p \colon A _q ^\ast \rightarrow T _p A ^\ast $, we can write
\begin{alignat*}{2}
  \bigl\langle \mathrm{d} \mathcal{L} _a ^{\mathrm{ver}} ( a, p ) , \overline{ \nabla } _a b  \bigr\rangle &= \bigl\langle \mathrm{d} \mathcal{L} ( a, p ) , V _a ( \overline{ \nabla } _a b ) \bigr\rangle &&= \frac{\mathrm{d}}{\mathrm{d}\epsilon}  \mathcal{L} ( a + \epsilon \overline{ \nabla } _a b , p ) \Bigr\rvert _{ \epsilon = 0 } \\
  \bigl\langle r, \mathrm{d} \mathcal{L} _p ^{\mathrm{ver}} ( a, p ) \bigr\rangle &= \bigl\langle V _p r , \mathrm{d} \mathcal{L} ( a, p ) \bigr\rangle &&= \frac{\mathrm{d}}{\mathrm{d}\epsilon} \mathcal{L} ( a, p + \epsilon r ) \Bigr\rvert _{ \epsilon = 0 } ,
\end{alignat*}
so $ \mathrm{d} \mathcal{L} _a ^{\mathrm{ver}} $ and
$ \mathrm{d} \mathcal{L} _p ^{\mathrm{ver}} $ are simply the fiber
derivatives of $\mathcal{L}$ along the $A$ and $ A ^\ast $ fibers,
respectively.

We now show that a particular choice of $\mathcal{L}$ gives a
variational principle equivalent to the generalized Lie--Poisson
equations for $H$.

\begin{theorem}
  \label{thm:variational}
  Given a Hamiltonian $ H \in C ^\infty ( A ^\ast ) $, an
  $ ( A \oplus A ^\ast ) $-path $ (a, p ) $ satisfies the $ (\pm) $
  generalized Lie--Poisson equations \eqref{e:lie-poisson_connection}
  with $ a = \mp \mathrm{d} H ^{\mathrm{ver}} (p) $ if and only if 
  \begin{equation*}
    \delta \int _I  \bigl( \langle p, a \rangle \pm H ( p) \bigr) \,\mathrm{d}t  = 0 ,
  \end{equation*}
  with respect to admissible variations.
\end{theorem}

\begin{proof}
  Equation \eqref{e:lie-poisson_connection_hor} is just the $A$-path
  condition for $a$, so it holds automatically for
  $ ( A \oplus A ^\ast )$-paths.  Let
  $ \mathcal{L} (a,p) = \langle p, a \rangle \pm H (p) $, and let
  $ \nabla $ be a $ T Q $-connection on $A$. Taking the fiber
  derivatives of $ \mathcal{L} $, as above, we see that
  \begin{equation*}
    \mathrm{d} \mathcal{L} _a ^{\mathrm{ver}} ( a, p ) = p , \qquad \mathrm{d} \mathcal{L} _p ^{\mathrm{ver}} ( a, p ) = a \pm \mathrm{d} H ^{\mathrm{ver}} (p) .
  \end{equation*}
  Next, to show that the horizontal derivative of
  $ \phi ( a, p ) \coloneqq \langle a, p \rangle $ vanishes, we use
  the horizontal lifts $ H _a \colon T _q Q \rightarrow T _a A $ of
  $ \nabla $ and $ H _p \colon T _q Q \rightarrow T _p A ^\ast $ of
  $ \nabla ^\ast $. Similarly to the proof of \cref{lem:splitting}, if
  $ \xi \in \Gamma (A) $ and $ \mu \in \Gamma ( A ^\ast ) $ satisfy
  $ \xi (q) = a $ and $ \mu (q) = p $, and if $ v \in T _q Q $, then
  \begin{align*}
    \bigl\langle \mathrm{d} \phi ^{\mathrm{hor}} ( a, p ) , v \bigr\rangle
    &= \bigl\langle \mathrm{d} \phi ( a, p ) , ( H _a v, H _p v ) \bigr\rangle \\
    &= \Bigl\langle \mathrm{d} \phi ( a, p ) , \bigl( T \xi (v) - V _a ( \nabla _v \xi ) , T \mu (v) - V _p ( \nabla _v ^\ast \mu ) \bigr) \Bigr\rangle \\
    &= \bigl\langle \mathrm{d} \phi ( a, p ) , T ( \xi, \mu ) (v) \bigr\rangle - \bigl\langle \mathrm{d} \phi _a ^{\mathrm{ver}} ( a, p ), \nabla _v \xi \bigr\rangle - \bigl\langle \nabla _v ^\ast \mu , \mathrm{d} \phi _p ^{\mathrm{ver}} (a,p) \bigr\rangle \\
    &= v \bigl[ \langle \xi, \mu \rangle \bigr] - \langle \mu , \nabla _v \xi \rangle  - \langle \nabla _v ^\ast \mu , \xi \rangle \\
    &= 0 ,
  \end{align*}
  where the last equality is the defining property of
  $ \nabla ^\ast $. (The preceding lines can be seen as the Leibniz
  rule for the covariant derivative of the tensor
  $ \phi \colon A \otimes A ^\ast \rightarrow \mathbb{R} $.)  Thus,
  \begin{equation*}
    \mathrm{d} \mathcal{L} ^{\mathrm{hor}} ( a, p ) = \pm \mathrm{d} H ^{\mathrm{hor}} (p) .
  \end{equation*}
  We conclude that \eqref{e:a+a*_elp} holds if and only if
  \eqref{e:lie-poisson_connection_ver} and
  $ a = \mp \mathrm{d} H ^{\mathrm{ver}} (p) $ hold.
\end{proof}

\begin{example}
  If $ A = \mathfrak{g} \rightarrow \bullet $ is a Lie algebra, the
  variational principle of \cref{thm:variational} can be restated as
  follows: Find
  $ ( \xi, \mu ) \colon I \rightarrow \mathfrak{g} \oplus \mathfrak{g}
  ^\ast $ such that
  \begin{equation*}
    \delta \int _I \bigl( \langle \mu , \xi \rangle \pm H (\mu) \bigr) \,\mathrm{d}t = 0 ,
  \end{equation*}
  where $ \delta \xi = \dot{ \eta } + [ \xi, \eta ] $ for arbitrary
  $ \eta \colon I \rightarrow \mathfrak{g} $ vanishing at the
  endpoints of $I$, and where $ \delta \mu $ is arbitrary with no
  boundary conditions. This recovers the \emph{Lie--Poisson
    variational principle} of \citet[Theorem 2.1]{CeMaPeRa2003}, who
  present it for the $(-)$ Lie--Poisson equations.
\end{example}

\begin{example}
  If $ A = T Q \rightarrow Q $ is the tangent bundle, recall that the
  $TQ$-path condition is $ a = \dot{q} $, since $\rho$ is the identity
  map. Thus, $TQ$-paths are identified (via tangent prolongation) with
  paths in $Q$, and it follows that $ ( T Q \oplus T ^\ast Q ) $-paths
  are identified with paths in $ T ^\ast Q $. Hence,
  \cref{thm:variational} gives the equivalence of the $ ( \pm ) $
  Hamilton's equations \eqref{e:cotangent_equations} with
  \begin{equation*}
    \delta \int _I \bigl( \langle p, \dot{q} \rangle \pm H ( q, p ) \bigr) \,\mathrm{d}t = 0 .
  \end{equation*}
  The $ (-) $ case, which is the usual sign convention, recovers
  \emph{Hamilton's phase space principle}.
\end{example}

\subsection{Special case: the Hamilton--Poincar\'e variational
  principle and equations}
\label{sec:hamilton-poincare}

If $ A = T Q / G \rightarrow Q / G $ is the Atiyah algebroid of a
principal bundle $ Q \rightarrow Q / G $, we now show that
\cref{thm:variational} recovers the equivalence of the
\emph{Hamilton--Poincar\'e variational principle} and
\emph{Hamilton--Poincar\'e equations} of \citet[Theorem
8.1]{CeMaPeRa2003}. This is similar to the Lie algebroid approach to the
Lagrange--Poincar\'e variational principle and Lagrange--Poincar\'e
equations in \citet*[Section 2.4]{LiStTa2017}, from which we adapt
some of the details. Note that the base of this algebroid is $ Q / G $
rather than $Q$.

As in \citep{LiStTa2017}, we begin by observing that a principal
connection is a right splitting of the Atiyah sequence
\begin{equation*}
  0 \rightarrow \widetilde{ \mathfrak{g}  } \rightarrow T Q / G \xrightarrow{ \rho } T ( Q / G ) \rightarrow 0 ,
\end{equation*}
where $ \widetilde{ \mathfrak{g} } $ denotes the adjoint bundle
$ Q \times _G \mathfrak{g} $, and where a left splitting is a
principal connection $1$-form (\citet[Chapter
5]{Mackenzie2005}). This gives a splitting of the Atiyah algebroid
$ A = T Q / G \cong T ( Q / G ) \oplus \widetilde{ \mathfrak{g} } $,
where $\rho$ is the projection onto the first component. In terms of
this splitting, the bracket of two sections
$ \xi = ( X , \overline{ \xi } ) $ and
$ \eta = ( Y, \overline{ \eta } ) $ is
\begin{equation*}
  \bigl[ ( X , \overline{ \xi } ) , ( Y, \overline{ \eta } ) \bigr] = \bigl( [X, Y], \widetilde{ \nabla } _X \overline{ \eta } - \widetilde{ \nabla } _Y \overline{ \xi } + [ \overline{ \xi } , \overline{ \eta } ] - \widetilde{ R } ( X, Y ) \bigr) ,
\end{equation*}
where $ \widetilde{ \nabla } $ is the covariant derivative and
$ \widetilde{ R } $ the curvature form of the principal connection
(\citet[Equation 3.4]{LiStTa2017}, \citet[Theorem
5.2.4]{CeMaRa2001}, \citet[Theorem 7.3.7]{Mackenzie2005}).

Given an $A$-path $ a = ( x, \dot{x} , \overline{ v } ) $, where $x$
is the base path in $ Q / G $, and an arbitrary path
$ b = ( x, \delta x , \overline{ \eta } ) $ in $A$, it follows from a
calculation in \citep[Section 3.4]{LiStTa2017} that admissible
variations have horizontal component $ \rho (b) = \delta x $ and
vertical component
\begin{equation}
  \label{e:nabla_a_b}
  \overline{ \nabla } _a b = \bigl( \overline{ \nabla } _{ \dot{x} } ( \delta x ) , \widetilde{ \nabla } _{ \dot{x} } \overline{ \eta } + [ \overline{ v }, \overline{ \eta } ] - \widetilde{ R } ( \dot{x}, \delta x ) \bigr) .
\end{equation}
That is, admissible variations have the form $ \delta a = ( \delta x , \delta \dot{x} , \delta \overline{ v } ) $, where
\begin{equation*}
  \delta \dot{x} = \overline{ \nabla } _{ \dot{x} } ( \delta x ) , \qquad \delta \overline{ v } = \widetilde{ \nabla } _{ \dot{x} } \overline{ \eta } + [ \overline{ v } , \overline{ \eta } ] - \widetilde{ R } ( \dot{x} , \delta x ) ,
\end{equation*}
and where $ \delta x $ and $ \overline{ \eta } $ both vanish at the
endpoints of $I$.  Finally, using the principal connection to split
$ A ^\ast = T ^\ast Q / G \cong T ^\ast ( Q / G ) \oplus \widetilde{
  \mathfrak{g} } ^\ast $, we can write
$ p = ( x, y , \overline{ \mu } ) $, whose variations have the form
$ \delta p = ( \delta x , \delta y , \delta \overline{ \mu } ) $,
where $ \delta y $ and $ \delta \overline{ \mu } $ are arbitrary.
It follows that the variational principle in \cref{thm:variational}
can be written as
\begin{equation*}
  \delta \int _I \bigl( \langle y, \dot{x} \rangle + \langle \overline{ \mu } , \overline{ v } \rangle \pm H ( x, y, \overline{ \mu } ) \bigr) \,\mathrm{d}t = 0 ,
\end{equation*}
subject to the admissible variations above. The $ ( - ) $ case is
precisely the Hamilton--Poincar\'e variational principle of
\citet[Theorem 8.1]{CeMaPeRa2003}.

Let us now see how this corresponds to the generalized Lie--Poisson
equations.  First, observe that \eqref{e:lie-poisson_connection_hor}
says that $ a = ( x , \dot{x}, \overline{ v } ) $ is an $A$-path, so
$ a = \mp \mathrm{d} H ^{\mathrm{ver}} (p) $ becomes
\begin{subequations}
  \label{e:hamilton-poincare}
  \begin{align}
    \dot{x} &= \mp \frac{ \partial H }{ \partial y } ,\\
    \overline{ v } &= \mp \frac{ \partial H }{ \partial \overline{ \mu } } .
  \end{align}
Next, observe that, for $ p = ( x, y, \overline{ \mu } ) $ and
arbitrary $ b = ( x, \delta x , \overline{ \eta } ) $, by
\eqref{e:nabla_a_b} we have
\begin{equation*}
  \langle p , \overline{ \nabla } _a b \rangle = \bigl\langle y, \overline{ \nabla } _{ \dot{x} } ( \delta x ) \bigr\rangle + \bigl\langle \overline{ \mu } , \widetilde{ \nabla } _{ \dot{x} } \overline{ \eta } + [ \overline{ v } , \overline{ \eta } ] - \widetilde{ R } ( \dot{x}, \delta x ) \bigr\rangle ,
\end{equation*}
from which it follows that
\begin{equation*}
  \overline{ \nabla } _a ^\ast p = \Bigl( \overline{ \nabla } _{ \dot{x} } ^\ast y + \bigl\langle  \overline{ \mu } , \widetilde{ R } ( \dot{x}, \cdot ) \bigr\rangle , \widetilde{ \nabla } ^\ast _{ \dot{x} } \overline{ \mu } - \operatorname{ad} _{ \overline{ v } } ^\ast \overline{ \mu } \Bigr).
\end{equation*}
Finally, equation \eqref{e:lie-poisson_connection_ver} says that this is equal to
\begin{equation*}
  \pm \rho ^\ast \bigl( \mathrm{d} H ^{\mathrm{hor}} (p ) \bigr) = \biggl( \pm \frac{ \partial H }{ \partial x } , 0 \biggr).
\end{equation*}
That is,
\begin{align}
  \overline{ \nabla } _{ \dot{x} } ^\ast y &= \pm \frac{ \partial H }{ \partial x } - \bigl\langle \overline{ \mu } , \widetilde{ R } ( \dot{x}, \cdot ) \bigr\rangle ,\\
  \widetilde{ \nabla } _{ \dot{x} } ^\ast \overline{ \mu } &= \operatorname{ad} ^\ast _{ \overline{ v } } \overline{ \mu } .
\end{align}
\end{subequations}
In the $ ( - ) $ case, \eqref{e:hamilton-poincare} is the
coordinate-free form of the Hamilton--Poincar\'e equations in
\citet{CeMaPeRa2003}, modulo small differences in notation, e.g.,
\citep{CeMaPeRa2003} write both covariant derivatives
$ \overline{ \nabla } _{ \dot{x} } $ and
$ \widetilde{ \nabla } _{ \dot{x} } $ as $ {D} / {D} t $.

\section{Correspondence to the Lagrangian case}
\label{sec:lagrangian}

We conclude with a discussion of the relationship between Hamiltonian
mechanics on $ A ^\ast $ and Lagrangian mechanics on $A$, linking the
results of the present paper to those of \citet*{LiStTa2017}.

\begin{definition}
  We say $ H \in C ^\infty ( A ^\ast ) $ is \emph{hyperregular} if
  $ \mathrm{d} H ^{\mathrm{ver}} \colon A ^\ast \rightarrow A $ is a
  diffeomorphism. Likewise, $ L \in C ^\infty (A) $ is hyperregular if
  $ \mathrm{d} L ^{\mathrm{ver}} \colon A \rightarrow A ^\ast $ is a
  diffeomorphism.
\end{definition}

\begin{remark}
  Here, $ \mathrm{d} H ^{\mathrm{ver}} $ and
  $ \mathrm{d} L ^{\mathrm{ver}} $ are defined independently of a
  choice of connection, since they are simply the derivatives along
  fibers. These fiber derivatives are often denoted (especially in the
  case $ A = T Q $) by $ \mathbb{F} H $ and $ \mathbb{F} L $ and
  called \emph{Legendre transformations}.
\end{remark}

If $H \in C ^\infty ( A ^\ast ) $ is a hyperregular Hamiltonian, and
$ A ^\ast $ is equipped with the $ ( \pm ) $ generalized Lie--Poisson
structure, we define the Lagrangian $ L \in C ^\infty (A) $ by
\begin{subequations}
  \label{e:H_to_L}
  \begin{equation}
    L ( a ) = \langle p, a \rangle \pm H ( p ) ,
  \end{equation}
  where $p \in A ^\ast $ is defined implicitly by
  $a = \mp \mathrm{d} H ^{\mathrm{ver}} (p) $. It follows by a short
  calculation using the chain rule that $L$ is also hyperregular with
  \begin{equation}
    p = \mathrm{d} L ^{\mathrm{ver}} (a) .
  \end{equation}
\end{subequations}
(See~\citet[Exercise 7.2-3]{MaRa1999} for the case of arbitrary vector
bundles.) Conversely, given a hyperregular Lagrangian
$ L \in C ^\infty (A) $, we may define $ H \in C ^\infty ( A ^\ast ) $
by
\begin{subequations}
  \label{e:L_to_H}
  \begin{equation}
    H (p) = \mp \bigl(  \langle p , a \rangle - L (a) \bigr) ,
  \end{equation}
  where $ a \in A $ is defined implicitly by
  $ p = \mathrm{d} L ^{\mathrm{ver}} (a) $. A calculation similar to the
  one described above shows that $H$ is also hyperregular with
  \begin{equation}
    a = \mp \mathrm{d} H ^{\mathrm{ver}} (p) .
  \end{equation}
\end{subequations}
Thus, the hypothesis of hyperregularity allows one to start with
either a Hamiltonian or Lagrangian and pass to the other.

The following theorem summarizes the main results of this paper, along
with those of \citet*{LiStTa2017}, and establishes their relationship.
Compare \citet[Theorem 8.1]{CeMaPeRa2003} for the special case of the
Atiyah algebroid discussed in \cref{sec:hamilton-poincare}.

\begin{theorem}
  Let $ ( a, p ) $ be an $ ( A \oplus A ^\ast ) $-path, and let
  $ \nabla $ be a $ T Q $-connection on $A$. Given a Hamiltonian
  $ H \in C ^\infty ( A ^\ast ) $, the following are equivalent:
  \begin{enumerate}[label=(\roman*)]
  \item The variational principle
    \begin{equation*}
      \delta \int _I \bigl( \langle p, a \rangle \pm H (p) \bigr) \,\mathrm{d}t = 0 
    \end{equation*}
    holds with respect to admissible variations of
    $ ( A \oplus A ^\ast ) $-paths.

  \item The generalized Lie--Poisson equations
    \begin{equation*}
      \overline{ \nabla } _a ^\ast p = \pm \rho ^\ast \bigl( \mathrm{d} H ^{\mathrm{hor}} (p) \bigr) 
    \end{equation*}
    hold with $ a = \mp \mathrm{d} H ^{\mathrm{ver}} (p) $.
  \end{enumerate}
  Given a Lagrangian $ L \in C ^\infty (A) $, the following are
  equivalent:
  \begin{enumerate}[resume, label=(\roman*)]
  \item The variational principle
    \begin{equation*}
      \delta \int _I L ( a ) \,\mathrm{d}t = 0 
    \end{equation*}
    holds with respect to admissible variations of $A$-paths, and
    $ p = \mathrm{d} L ^{\mathrm{ver}} (a) $.

  \item The Euler--Lagrange--Poincar\'e equations
    \begin{equation*} 
      \overline{ \nabla } _a ^\ast p  = \rho ^\ast \bigl( \mathrm{d} L ^{\mathrm{hor}} (a) \bigr) 
    \end{equation*} 
    hold with $ p = \mathrm{d} L ^{\mathrm{ver}} (a) $.
  \end{enumerate}
  Under the hypothesis of hyperregularity, all of (i)--(iv) are
  equivalent.
\end{theorem}

\begin{remark}
  In \citep{LiStTa2017}, $ \overline{ \nabla } ^\ast _a $ does not
  denote the dual connection but rather the adjoint of
  $ \overline{ \nabla } _a $, similar to $ \operatorname{ad} $ and
  $ \operatorname{ad} ^\ast $. This is
  $ - \overline{ \nabla } ^\ast _a $ in the notation of the present
  paper, leading to a change of sign in the
  Euler--Lagrange--Poincar\'e equations.
\end{remark}

\begin{proof}
  The equivalence of (i) and (ii) is \cref{thm:variational}, and that
  of (iii) and (iv) is \citep[Theorem 3.12]{LiStTa2017}. Under the
  hypothesis of hyperregularity, (i) implies (iii) by
  \eqref{e:H_to_L}, and conversely, (iii) implies (i) by
  \eqref{e:L_to_H}.
\end{proof}

\end{document}